\documentclass[a4paper, 11pt, twoside, final]{scrartcl}

\usepackage{preamble}
\usepackage{command}
\usepackage{hyphenation}
\usepackage{mathdots}
\usepackage{cite}

\AUTHOR{Robert \pers{Beinert}\Inst{1} and Marzieh
  \pers{Hasannasab}\Inst{1}}
\SHORTAUTHOR{R.~Beinert and M.~Hasannasab}
\TITLE{Phase Retrieval via Polarization in Dynamical Sampling}
\SHORTTITLE{Phase retrieval in dynamical sampling}
\INSTITUTE{\Inst{1}\parbox[t]{0.95\linewidth}{Institut für Mathematik\\
    Technische Universität Berlin\\
    Straße des 17. Juni 136\\
    10623 Berlin, Germany
  }
}
\CORRESPONDENCE{R. \pers{Beinert}: \email{beinert@math.tu.berlin.de}\\
  M. \pers{Hasannasab}: \email{hasannas@math.tu-berlin.de}}
\ABSTRACT{In this paper we consider the nonlinear inverse problem of phase
  retrieval in the context of dynamical sampling.  Where phase
  retrieval deals with the recovery of signals \& images from
  phaseless measurements, dynamical sampling was introduced by
  Aldroubi et al in 2015 as a tool to recover diffusion fields from
  spatiotemporal samples.  Considering finite-dimensional signals
  evolving in time under the action of a known matrix, our aim is to
  recover the signal up to global phase in a stable way from the
  absolute value of certain space-time measurements. First, we state
  necessary conditions for the dynamical system of sampling vectors to
  make the recovery of the unknown signal
  possible. The conditions deal with the spectrum of the given matrix
  and the initial sampling vector.  Then, assuming that
  we have access to a specific set of further measurements related to
  aligned sampling vectors, we provide a
  feasible procedure to recover almost every signal up to global
  phase using polarization techniques.  Moreover, we show that by
  adding extra conditions like full spark, the recovery of all signals is
  possible without exceptions.}

\begin{document}
\thispagestyle{plain}
\TitleHeader

\section{Introduction}
\label{sec:introduction}
Phase retrieval was introduced in \cite{patterson1934fourier} as a
problem of reconstructing a signal from its Fourier magnitude and has
become increasingly popular in image and signal processing due to its
applications in crystallography \cite{Hau91,KH91,millane1990phase},
astronomy \cite{BS79,fienup1987phase}, and laser optics
\cite{SSD+06,SST04}.  In all these applications, phase retrieval
occurs as ill-posed inverse problem, where the tremendous
ambiguousness is the most critical point.  For the classical problem,
the non-uniqueness has been well studied, and there are several
approaches to surmount this issue by enforcing a priori assumptions or
exploit additional measurements, see for instance
\cite{ADGY19,BP15,BP20,BBE17,BS79,grohs2019mathematics,HHLO83,KK14,KST95,SECC+15}
and references therein.  Moreover, phase retrieval also occurs in the
more abstract setting of frames, where the unknown image has to be
recovered from the magnitudes of its frame coefficients.  For generic
and specific frames, this problem has been studied in
\cite{alexeev2014phase,alexeev2012full,BBCE09,balan2006signal,CESV13,CSV13}.
The purpose of the current paper is to consider phase retrieval in the
setting of dynamical sampling \cite{ACMT17,AHP19,AK16,AP17}, which
originate back to sampling and recovering diffusion fields form
spatiotemporal measurements \cite{LV09,RCLV11}.

In dynamical sampling, we consider an unknown vector
$\Vek x\in \BC^d$ that evolves under the action of a matrix
$\Mat A\in \BC^{d\times d}$ meaning that at time $\ell\in\BN$ the
signal becomes $\Vek x_\ell = (\Mat A^*)^\ell \Vek x$.   Our aim
is to recover  $\Vek x$ up to global phase from
phaseless measurements.  More precisely, we want to recover $\Vek x$ form
\begin{equation}\label{def:measurements}
  |\langle (\Mat A^*) ^\ell\Vek x ,\Vek \phi \rangle |
  =
  |\langle \Vek x , \Mat A^\ell \Vek \phi \rangle |,
\end{equation}
where $\ell = 0, \dots, L-1$ with $L\geq d$, and where
$\Vek \phi \in \BC^d$ is some sampling vector.  Phase retrieval in
dynamical sampling has been already considered for real Hilbert spaces
in \cite{aldroubi2020phaseless,aldroubi2017phase}, where the authors
provided conditions to ensure that
$\{\Mat A^{\ell} \Vek\phi\}_{\ell=0}^{ L-1} $ has the complementary
property meaning that each subset or its complement spans the entire
space.  The results have then be generalized to several sampling
vectors.  The complementary property is here equivalent to the
uniqueness (up to global phase) of phase retrieval from
\eqref{def:measurements}.  However, their techniques cannot be
immediately generalized to the complex setting since here the
complementary property is not sufficient
\cite{balan2006signal,BCMN14}.

To insure that we can do phase retrieval in $\BC^d$, we will assume
that $\{\Mat A^\ell \Vek \phi\}_{\ell=0}^{L-1}$ is a frame, and we
will align $\Vek \phi$ with specifically chosen additional sampling
vectors to exploit polarization techniques.  This idea is inspired by
interferometry used in \cite{alexeev2014phase,Bei17b}.  Using the
extra information, we first recover the frame coefficients
$\langle \Vek x , \Mat A^\ell\Vek \phi\rangle$ up to global phase
and then recover $\Vek x$ in a stable way via the dual frame of
$\{\Mat A^\ell \Vek \phi\}_{\ell=0}^{L-1}$.

The paper is organized as follows. In Section \ref{sec:preleminaries}
we set the stage by providing the necessary background information
about polarization identities, frames and Vandermonde matrices. In
Section \ref{sec:dynamical-frames} we find conditions on the spectrum
of $\Mat A$ and the vector $\Vek \phi$ such that the iterated set
$\{ \Mat A^\ell \Vek \phi\}_{\ell=0}^{L-1}$ is a frame. Moreover, in
Section~\ref{sec: Full_spark_dynamical_frames}, we provide conditions
under which this frame has full spark. In Section
\ref{sec:phase-retr-dynam} we prove that the aligned sampling vectors
allow phase retrieval for almost all $\Vek x\in\BC^d$; moreover, if
the underlying dynamical frame has full spark, the recovery of all
$\Vek x\in\BC^d$ is possible.

\section{Preleminaries}
\label{sec:preleminaries}
\subsection{Polarization and Relative Phases}
\label{subsec:polar-relat-phas}

Our main results are based on the following polarization technique,
which allow the recovery of the lost phases from certain phaseless
information.

\begin{theorem}[Polarization, \cite{Bei17b}]
  \label{thm:polarization}
  Let $\alpha_1, \alpha_2 \in \BR$ satisfy $\alpha_1 -
  \alpha_2 \not\in \pi \BZ$.  Then, for every $z_1, z_2 \in
  \BC \setminus \{0\}$, the product $\bar z_1 z_2$ is uniquely determined by
  \begin{equation*}
    \absn{z_1}, \quad
    \absn{z_2}, \quad
    \absn{z_1 + \e^{\I\alpha_1} \, z_2}, \quad
    \absn{z_1 + \e^{\I\alpha_2} \, z_2}.
  \end{equation*}
\end{theorem}

\begin{proof}
  On the basis of the polar decomposition
  $z_\ell = \absn{z_\ell} \, \e^{\I \phi_\ell}$ with
  $\ell \in \{1,2\}$, the last two absolute values are equivalent to
  \begin{equation*}
    \absn{z_1
      + \e^{\I \alpha_\ell} \, z_2}^2
    = \absn{z_1}^2
    + \absn{z_2}^2
    + 2 \absn{z_1}
    \absn{z_2}
    \Re [\e^{\I ( \phi_2 - \phi_1 + \alpha_\ell )}].
  \end{equation*}
  Since $z_1$ and $z_2$ are non-zero, we can thus extract the real parts
  \begin{equation*}
    r_\ell \coloneqq \Re \bigl[ \e^{\I ( \phi_2 - \phi_1 + \alpha_\ell)} \bigr].
  \end{equation*}
  Using Euler's formula, we obtain the linear equation system
  \begin{align*}
    r_1
    &= \cos( \alpha_1 )
      \cos( \phi_2 - \phi_1)
      - \sin( \alpha_1)
      \sin( \phi_2 - \phi_1)
    \\
    r_2
    &= \cos( \alpha_2 )
      \cos( \phi_2 - \phi_1)
      - \sin( \alpha_2)
      \sin( \phi_2 - \phi_1).
  \end{align*}
  The determinant of the system matrix is here
  \begin{equation*}
    \det
    \begin{psmallmatrix}
      \cos \alpha_1 & -\sin \alpha_1 \\
      \cos \alpha_2 & -\sin \alpha_2
    \end{psmallmatrix}
    = \sin \mleft( \alpha_1 - \alpha_2 \mright),
  \end{equation*}
  which is non-zero by assumption; so $\cos(\phi_2-\phi_1)$ and
  $\sin(\phi_2 - \phi_1)$ are uniquely determined by the given data.
  Knowing the relative phase $\phi_2 - \phi_1$, we calculate
  $\bar z_1 z_2$.  \qed
\end{proof}

\begin{remark}[Real polarization] 
  For every $z_1,z_2 \in \BR \setminus \{0\}$, the product $z_1 z_2$
  is uniquely determined by $\absn{z_1}$, $\absn{z_2}$, and
  $\absn{z_1 +\alpha z_2}$ with $\alpha \in \{-1,1\}$ because
  $\absn{z_1 +\alpha z_2}^2 = z_1^2 + z_2^2 + 2\alpha z_1 z_2$.
\end{remark}

\begin{remark}[Polarization identities\cite{alexeev2014phase,Bei17b}]
  For certain $\alpha_1$ and $\alpha_2$ as in
  Theorem~\ref{thm:polarization}, the phase of $z_1, z_2$ can be
  computed without solving a linear equation system.  More generally,
  if $\zeta_K$ is chosen to be the $K$th root of unity, then we have
  \begin{equation*}
    \bar z_1 z_2
    = \frac{1}{K} \sum_{k=0}^{K-1} \zeta_K^k \absn{ z_1 +
      \zeta_K^{-k} \, z_2}^2.
  \end{equation*}
\end{remark}

\subsection{Frames}
\label{subsec:Frames}

Given a matrix $\Mat A\in \BC^{d\times d}$ and a vector
$\Vek \phi\in \BC^d$, the set
$\{\Mat A^\ell \Vek \phi\}_{\ell=0}^{L-1}$ is called a \emph{dynamical
  frame} if it spans $\BC^{d}$. Dynamical frames were first introduced
in \cite{ACMT17} in order to recover a signal evolving in time from
certain time-space measurements, where also the infinite dimensional
problem is addressed. The topic was further developed in
\cite{AHP19,AP17,AK16,cabrelli2020dynamical,christensen2017operator,christensen2018dynamical,christensen2020frame,aldroubi2017iterative,philipp2017bessel,RCLV11}. An
arbitrary vector $\Vek x\in\BC^d$ can be recovered from the set
$\{ \langle \Mat A^\ell\Vek x ,\Vek \phi \rangle \}_{\ell=0}^{L-1}$ in
an stable way if there exists $\alpha,\beta>0$ such that
\[	\alpha\| \Vek y\|^2 \leq \| \{  \langle \Mat A^\ell\Vek y ,\Vek \phi \rangle \}_{\ell=0}^{L-1}\|^2 \leq \beta \|\Vek y\|^2, \quad \text{for all } \Vek y\in \BC^d,
\]
i.e., when  the set $\{ ( \Mat A^T)^\ell\Vek \phi  \}_{\ell=0}^{L-1}$ is a frame for $\BC^d$.

For any dynamical frame $\{\Mat A^\ell \Vek \phi\}_{\ell=0}^{L-1}$ there exists a set of vectors in the form $\{\Mat B^\ell \tilde{\Vek\phi}\}_{\ell=0}^{L-1}$ such that every $\Vek x\in \BC^d$ can be written as 
\begin{equation}\label{def: reconstruction_formula}
	\Vek x = \sum_{\ell=0}^{L-1} \langle  \Vek x, \Mat A^\ell \Vek \phi \rangle \Mat B^\ell \Vek \tilde\phi.
\end{equation}
Indeed for a given frame $\{\Mat A^\ell \Vek \phi\}_{\ell=0}^{L-1}$,
the frame matrix
$\Mat T \coloneqq \sum_{\ell=0}^{L-1}  \Mat A^\ell
\Vek \phi (\Mat A^\ell \Vek \phi)^* $ is symmetric and positive
definite and the canonical dual frame
$\{\Mat T^{-1}\Mat A^\ell \Vek \phi\}_{\ell=0}^{L-1}$ can be written
in the form $\{\Mat B^\ell \tilde{\Vek\phi}\}_{\ell=0}^{L-1}$ where
$\Mat B \coloneqq \Mat T^{-1}\Mat A\Mat T$ and
$\tilde{\Vek \phi}\coloneqq\Mat T^{-1}\Vek \phi$. For more information
about frames, we refer to \cite{christensen2016introduction}.

\begin{example}
  Let $d=2$, and consider the rotation matrix
  $\Mat A = \bigl(\begin{smallmatrix}
    \cos\theta & -\sin\theta \\
    \sin\theta &\cos\theta
	\end{smallmatrix}\bigr)$
	with $\theta\neq k\pi$ for $ k\in\BZ$.  For every nonzero
        vector $\Vek\phi\in \BC^2$ and $L\geq 2$ the set
        $\{ \Mat A^\ell \Vek\phi\}_{\ell=0}^{L-1}$ is a dynamical
        frame for $\BC^2$.
\end{example}

\subsection{Vandermonde Matrices}

As we will see in Section \ref{sec:dynamical-frames}, the frame
property of the set $\{\Mat A^\ell\Vek \phi\}_{\ell=0}^{L-1}$ is
highly related to the Vandermonde matrix generated by the  vector
$\Vek\lambda$ whose coordinates consists  of the eigenvalues of the
matrix $\Mat A$. There are different types of Vandermonde matrices in
the literature. We will need the following kinds. 

\paragraph{The Classical Vandermode Matrix} For
$\Vek \lambda \coloneqq (\lambda_0,\dots,\lambda_{d-1})^\T \in \BC^d$,
the Vandermonde matrix $\Mat V_{\Vek\lambda}\in \BC^{d\times L}$
generated by $\Vek \lambda$ is defined as
\begin{equation*}
	\Mat V_{\Vek \lambda} := ( \lambda_k^{\ell})_{k,\ell=0}^{d-1,L-1}.
\end{equation*}
The determinant of  a square Vandermonde matrix $\Mat V_{\Vek\lambda}\in \BC^{d\times d}$  equals to 
\[
\det\Mat V_{\Vek \lambda} = \prod_{0\leq k < j\leq d-1}(\lambda_k - \lambda_j).
\]

\paragraph{Generalization of the First Kind}
A generalized Vandermonde matrix of the first kind is a matrix
consisting of selective columns of $\Mat V_{\Vek \lambda}$. More
precisely, for a vector
$\Vek \lambda \coloneqq (\lambda_0,\dots,\lambda_{d-1})^\T \in \BC^d$
and $\Vek m \coloneqq (m_0, \dots, m_{L-1})^\T \in \BN_0^L$, the
Vandermonde matrix $\Mat V_{\Vek \lambda,\Vek m}\in \BC^{d\times L}$
is defined as
\begin{equation*}
	\Mat V_{\Vek \lambda, \Vek m}
	\coloneqq \bigl( \lambda_k^{m_\ell} \bigr)_{k,\ell=0}^{d-1,L-1},
\end{equation*}
The  Vandermonde determinant of the first kind may be
factorized by
\begin{equation}
  \label{eq:gen-vand-det}
  \det \Mat V_{\Vek \lambda, \Vek m}
  = \biggl( \prod_{k > j} (\lambda_k - \lambda_j) \biggr) \,
  S(\Vek \lambda).
\end{equation}
where $S$ is a symmetric polynomial in $\Vek \lambda$ with
non-negative, integer coefficients \cite{DVB08}.  The occurring
polynomials $S$ are better known as Schur functions \cite{Mac92,Mac95}.

\paragraph{Generalization of the Second Kind}
The second kind generalized Vandermonde matrix
$\widetilde{\Mat V}_{\Vek \lambda, \Mat m}\in \BC^{d\times L}$ is
defined as
\begin{equation*}
  \widetilde{\Mat V}_{\Vek \lambda, \Mat m}
  \coloneqq
  \begin{bmatrix}
    \Mat R_0 \\ \vdots \\ \Mat R_{M-1}
  \end{bmatrix}
  \qquad\text{with}\qquad
  \Mat R_j \coloneqq \biggl( \binom \ell k \, \lambda_j^{\ell - k}
  \biggr)_{k,\ell = 0}^{m_j-1,L-1},
\end{equation*}
where $M\in\BN$, $\Vek \lambda \in \BC^{M}$ and $\Vek m \in \BN^{M}$
such that $\absn{\Vek m} \coloneqq \sum_{j=0}^{M-1} \absn{m_j} =
d$. Clearly if $\Vek m$ is the unite vector, i.e.,
$\Vek m=(1,\dots,1)^T\in \BN^M $, then
$\widetilde{\Mat V}_{\Vek \lambda,\Vek m}$ equals the Vandermonde
matrix $\Mat V_{\Vek \lambda}$.  The determinant is given by
\begin{equation}
  \label{eq:general-vand-det:1}
  \det (\widetilde{\Mat V}_{\Vek \lambda, \Vek m})
  = \smashoperator{\prod_{0 \le k < j \le M-1}} ( \lambda_j -
  \lambda_k)^{m_k m_j},
\end{equation}
see \cite{Ait44,Kal84}.  Obviously, a square Vandermonde matrix
$ \widetilde{\Mat V}_{\Vek \lambda, \Vek m}$ is invertible precisely when 
$\Vek \lambda$ has distinct elements.

\begin{example}
  For $\Vek \lambda \coloneqq (\lambda_0, \lambda_1, \lambda_2)^\T$,
  $\Vek m \coloneqq (3,1,2)^\T$ and $L=d=6$, we have
  \begin{equation*}
    \widetilde{\Mat V}_{\Vek \lambda, \Vek m} \coloneqq
    \begin{bmatrix}
      1 & \lambda_0 & \lambda_0^2 & \lambda_0^3 & \lambda_0^4 &
      \lambda_0^5 \\
      0 & 1 & 2 \, \lambda_0 & 3 \, \lambda_0^2 & 4 \, \lambda_0^3 & 5
      \, \lambda_0^4 \\
      0 & 0 & 1 & 3 \, \lambda_0 & 6 \, \lambda_0^2 & 10 \, \lambda_0^3
      \\ 
      1 & \lambda_1 & \lambda_1^2 & \lambda_1^3 & \lambda_1^4 &
      \lambda_1^5 \\
      1 & \lambda_2 & \lambda_2^2 & \lambda_2^3 & \lambda_2^4 &
      \lambda_2^5 \\
      0 & 1 & 2 \, \lambda_2 & 3 \, \lambda_2^2 & 4 \, \lambda_2^3 & 5
      \, \lambda_2^4 \\
    \end{bmatrix}.
  \end{equation*}
\end{example}

\section{Dynamical Frames}
\label{sec:dynamical-frames}

To recover a signal from \eqref{def:measurements},
we first study conditions on the matrix
$\Mat A\in\BC^{d\times d}$ and the vector $\Vek \phi\in\BC^d$ such
that $\{\Mat A^\ell \Vek\phi\}_{\ell=0}^{L-1}$ is a frame for
$\BC^d$. The cornerstone is here the Jordan canonical form of $\Mat A$.
More precisely, every matrix $\Mat A \in \BC^{d \times d}$ is similar
to a so-called Jordan matrix meaning that there exists an
invertible matrix $\Mat S \in \BC^{d \times d}$ such that
$\Mat A = \Mat S \Mat J \Mat S^{-1}$ and $\Mat J \in \BC^{d \times d}$
is a blocked diagonal matrix of the
form 
\begin{equation*}\label{Jordan-form}
	\Mat J= \text{diag}(\Mat J_0,\dots, \Mat J_{M-1})
        \qquad\text{with}\qquad
	\Mat J_j= 
	\begin{pmatrix}
		\lambda_j&1&&\\
		&\lambda_j&\ddots&\\
		&&\ddots&1\\
		&&&\lambda_j
	\end{pmatrix}
	\in \BC^{m_j \times m_j},
\end{equation*}
where $\lambda_j$ is the $j$th eigenvalue and $m_j$ the corresponding
algebraic multiplicity, and where the columns of
$\Mat S = [\Vek S_0| \dots |\Vek S_{M-1}]$ with blocks
$ \Mat S_j = [\Vek s_{j,0}| \dots| \Vek s_{j,m_{j}-1}]$ span the
generalized eigenspaces of $\Mat A$.  Further, we have
$(\Mat A-\lambda_j\Mat I)^{k+1}\Vek s_{j,k}=0$ but
$(\Mat A-\lambda_j\Mat I)^{k}\Vek s_{j,k}\neq 0$ for
$k=0,\dots,m_j-1$.  The Jordan chain $\Mat S_j$ related to $\lambda_j$
is generated by $\Vek s_{j,m_j-1}$ via
$\Vek s_{j,k} = (\Mat A-\lambda\Mat I)^{m_j-k-1}\Vek s_{j,m_j-1}$.  We
say that $\Vek \phi$ depends on the Jordan generator or leading
generalized eigenvector $\Vek s_{j,m_j-1}$ if
$(\Mat S^{-1}\Vek \phi)_{k-1} \neq 0$ where $k=\sum_{i=0}^{j-1} m_i$.
For pairwise distinct eigenvalues $\lambda_j$ as usually assumed in
the following, the generators are unique up to scaling.  In this case,
$\Mat S$ is unique up to scaling and permutation of the blocks
$\Mat S_j$.  Finally we notice that the $\ell$th power of a Jordan
matrix and the corresponding Jordan blocks are given by
\begin{equation*}
	\Mat J^\ell= \text{diag}(\Mat J^\ell_0,\dots, \Mat J^\ell_{M-1})
        \qquad\text{with}\qquad
	\Mat J_j^\ell = \biggl( \binom \ell {n - k} \, \lambda_j^{\ell-n+k}
	\biggr)_{k,n= 0}^{m_j - 1}.
\end{equation*}

The following two theorems are special cases of \cite{ACMT17}, where
the construction of a frame by iterated actions of $\Mat A$ on a
finite set of sampling vectors $\{\Vek\phi_j\}\subset \BC^d$ is
studied.  In difference to \cite{ACMT17}, we provide brief, direct proofs
based on the Vandermonde determinant for the case that $\Mat A$ acts
on a single generator $\Vek \phi$.

\begin{theorem}[Dynamical basis]
	\label{thm:dyn-basis}
	Let $\Mat A \in \BC^{d \times d}$ be arbitrary.  Then
	$\{\Mat A^\ell \Vek \phi \}_{\ell = 0}^{d-1}$ is a basis if and only
	if the eigenvalues of the Jordan blocks of $\Mat A$ are pairwise
	distinct and  $\Vek \phi$ depends on all Jordan generators.
\end{theorem}

\begin{proof}
  Assume that $\Mat A$ has the Jordan decomposition
  $\Mat A = \Mat S \Mat J \Mat S^{-1}$.  We represent the vector
  $\Vek \phi$ with respect to the column-wise basis in $\Mat S$
  according to the size of the Jordan blocks in $\Mat J$.  More
  precisely, we denote by $\Vek \psi_j$ the coordinates corresponding
  to the basis vectors in $\Mat S_j$.  The coefficients are thus given
  by
  \begin{equation*}
    \Vek \psi \coloneqq
    \begin{pmatrix}
      \Vek \psi_0 \\ \vdots \\ \Vek \psi_{M-1}
    \end{pmatrix}
    = \Mat S^{-1} \Vek \phi.
  \end{equation*}
	Next, we consider the generated vectors
	$\Vek \phi_\ell \coloneqq \Mat A^\ell \Vek \phi$ with
	$\ell = 0, \dots, d-1$.  On the basis of the Jordan canonical
        form,
	they are given by $\Vek \phi_\ell = \Mat S \Mat J^\ell \Vek \psi$.
	Considering only the $j$th Jordan block, we notice
        \begin{equation}
		\label{eq:J_power_ell}
		\Mat J_j^\ell \Vek \psi_j
		= \Mat H(\Vek \psi_j) \,
		\biggl( \binom \ell k \, \lambda_j^{\ell -
			k} \biggr)_{k = 0}^{m_j-1},
	\end{equation}
	where 
	\begin{equation*}
		\Mat H(\Vek \psi_j)
		= 
		\begin{bmatrix}
			(\Vek \psi_j)_0 &(\Vek \psi_j)_1 & \dots & (\Vek\psi_j)_{m_j-1} \\
                        \vdots &\vdots&\iddots&\vdots\\
			(\Vek\psi_j)_{m_j-2}&(\Vek\psi_j)_{m_j-1}&&\\
			(\Vek\psi_j)_{m_j-1} &0& \dots & 0
		\end{bmatrix},
	\end{equation*}
	is an upper-left Hankel matrix in $\BC^{m_j\times m_j}$. The vector on the right-hand side of \eqref{eq:J_power_ell} is here the $\ell$th column of
	$\Mat R_j$ within the definition of generalized Vandermonde matrix
	$\widetilde{\Mat V}_{\Vek \lambda, \Vek m}$.  The matrix of the
	generated vectors may hence be written as
	\begin{equation*}
		[\Vek \phi_0 | \dots | \Vek \phi_{d-1}]
		= \Mat S
		\begin{bmatrix}
			\Mat H(\Vek \psi_0) & \dots & \Mat 0 \\
			\vdots & \ddots & \vdots \\
			\Mat 0 & \dots & \Mat H(\Vek \psi_{m-1})
		\end{bmatrix}
		\widetilde{\Mat V}_{\Vek \lambda, \Vek m}.
	\end{equation*}
	This matrix is invertible if and only if the generalized Vandermonde
	matrix $\widetilde{\Mat V}_{\Vek \lambda, \Vek m}$ is invertible, \ie\
	if the eigenvalues are pairwise distinct, see
	\eqref{eq:general-vand-det:1}, and if the Hankel matrices
	$\Mat H(\Vek \psi_j)$ are regular, \ie\ if the coefficients
	$(\Vek \psi_j)_{m_j-1}$ of the highest-order generalized
	eigenvectors do not vanish.  \qed
\end{proof}

\begin{theorem}[Dynamical frame]
	Let $L \ge d$, and let $\Mat A \in \BC^{d \times d}$ be arbitrary.  Then
	$\{\Mat A^\ell \Vek \phi \}_{\ell = 0}^{L-1}$ is a frame if and only
	if  the eigenvalues of the Jordan blocks of $\Mat A$ are pairwise
	distinct and  $\Vek \phi$ depends on all Jordan generators.
\end{theorem}

\begin{proof}
  If the vector $\Vek \phi$ is independent of one Jordan generator,
  then the images $\Mat A^\ell \Vek \phi$ are also independent of this
  generator; so $\{\Mat A^\ell \Vek \phi \}_{\ell = 0}^{L-1}$ can not
  be a frame for $\BC^d$.  Now assume that some eigenvalues of
  $\Mat A$ coincide, \ie\ the Jordan block to this eigenvalue
  decompose into several smaller Jordan blocks.
  Assume that the eigenvalues $\lambda_{j_0}$ and $\lambda_{j_1}$ coincide,
  and that the corresponding Jordan blocks have dimension
  $m_{j_0} \times m_{j_0}$ and $m_{j_1} \times m_{j_1}$. 
  Using the notation in the proof of Theorem~\ref{thm:dyn-basis},
    the coordinates of $\Vek \phi$ in
  $E \coloneqq \Span\{\Vek s_{j_0,m_{j_0-1}}, \Vek s_{j_1,m_{j_1-1}}\}$
  are $(\Vek \psi_{j_0})_{m_{j_0-1}}$ and $(\Vek \psi_{j_1})_{m_{j_1-1}}$.
  Applying $\Mat A^\ell$ to $\Vek \phi$, we get the coordinates
  $\lambda_{j_0}^\ell (\Vek \psi_{j_0})_{m_{j_0-1}}$ and
  $\lambda_{j_1}^\ell (\Vek \psi_{j_1})_{m_{j_1-1}}$ with
  $\lambda_{j_0} = \lambda_{j_1}$ regarding the subspace $E$.  Thus
  the projections $\proj_E (\{\Mat A^\ell \Vek \phi\}_{\ell =0}^{L-1})$ only span a
  one-dimensional subspace.  As a consequence
  $\{\Mat A^\ell \Vek \phi \}_{\ell =0}^{L-1}$ cannot span $\BC^d$,
  and we cannot obtain a frame.  The opposite direction has already be
  proven with Theorem~\ref{thm:dyn-basis}.  \qed
\end{proof}

Since  generic matrices $\Mat A \in \BC^{d \times d}$ are
diagonalizable with pairwise distinct eigenvalues, for almost all
matrices holds the following special case.

\begin{corollary}[Dynamical frame]\label{Cor:dynamical frame}
	Let $L \ge d$, and let $\Mat A \in \BC^{d \times d}$ be diagonalizable.  Then
	$\{\Mat A^\ell \Vek \phi \}_{\ell=0}^{L-1}$ is a frame if and only
	if the eigenvalues of $\Mat A$ are pairwise distinct and 
	$\Vek \phi$ depends on all eigenvectors.
\end{corollary}

\begin{proof}
	Since the Jordan blocks here reduces to size $1\times1$, the matrix
	of the generated vectors in the proof of  Theorem~\ref{thm:dyn-basis}
	simplifies to
	\begin{equation*}
		[\Vek \phi | \Mat A \Vek \phi | \dots | \Mat A^{d-1} \Vek \phi ]
		= \Mat S \diag(\Vek \psi) \, \Mat V_{\Vek \lambda}.
	\end{equation*}
	This matrix is invertible if and only if the classical Vandermonde
	matrix $\Mat V_{\Vek \lambda}\in \BC^{d\times d}$ is invertible  and none of the coordinates of $\Vek \psi$ 
	vanishes. \qed
\end{proof}

For $\Vek a \in \BC^d$, let $\Circ (\Vek a)$ denote the circulant matrix whose first column is given by the vector $\Vek a$. All circulant  matrices  are  diagonalizable with respect to the discrete Fourier transform, i.e., 
\[
\Circ(\Vek a) = \tfrac1{d} \, {\Mat F} \, \diag(\hat{\Vek a}) \,{ \Mat F}^{-1},\quad 
\]
where  $\hat{\Vek a} =\Mat F\Vek a $ is given via the Fourier matrix $\Mat F
=(\e^{-\frac{2\pi \I j k}{d}})_{j,k=0}^{d-1}$.

\begin{corollary}[Repeated convolution]
	Let $L \ge d$, and let $\Vek \phi, \Vek a \in \BC^d$ be arbitrary.  Then the family
	\begin{equation*}
		\bigl\{ \underbracket{\Vek a * \cdots * \Vek a}_{\ell \,
			\text{times}}  * \, \Vek \phi \bigr\}_{\ell=0}^{L-1}
	\end{equation*}
	is a frame for $\BC^d$ if and only if the coordinates of $\hat{\Vek \phi}$ do not
	vanish and the coordinates of $\hat{\Vek a}$ are pairwise
	distinct.
\end{corollary}
\begin{proof}
	Note that  $ \Vek a  * \, \Vek \phi = \Circ (\Vek a) \Vek \phi$  and $\Vek A \coloneqq \Circ(\Vek a)$ is a diagonalizable matrix that by hypothesis has pairwise distinct eigenvalues $\{ \hat{a}_k\}_{k=0}^{d-1}$. The result follows now from Corollary \ref{Cor:dynamical frame}.\qed 
      \end{proof}

\section{Full-Spark Dynamical Frames}
\label{sec: Full_spark_dynamical_frames}

A frame $\{ \Vek f_k\}_{k=0}^{L-1}$ has full spark if
every subset embracing $d$ elements spans
$\BC^d$. 
This property makes full-spark frames attractive in phase retrieval and
more generally in signal processing
\cite{alexeev2014phase,malikiosis2020full,balan2006signal,alexeev2012full}.
In the following, we study conditions ensuring that
frames generated via diagonalizable matrices have full spark. We show
that a dynamical frame has full spark precisely when the Vandermonde matrix
$\Mat V_{\Vek\lambda}$ related to the eigenvalues of $\Mat A$
has full spark.
\begin{theorem}
	\label{thm:dyn-full-spark-frame}
	Let $\Mat A \in \BC^{d \times d}$ be diagonalizable with eigenvalues
        $\Vek \lambda$. For every
        $L \ge d$, the set $\{\Mat A^\ell \Vek \phi\}_{\ell=0}^{L-1}$
        is a full spark frame if and only if $\Vek \phi$ depends on
        all eigenvectors and 
        $\Mat V_{\Vek \lambda}\in \BC^{d\times L}$
        has full spark.  
\end{theorem} 

\begin{proof}
	Assume that $\Mat A$ has the eigenvalue decomposition
	$\Mat A = \Mat S \Mat J \Mat S^{-1}$, where $\Mat J$ is a diagonal matrix, and denote the
	coordinates of $\Vek \phi$ with respect to $\Mat S$ by
	$\Vek \psi \coloneqq \Mat S^{-1} \Vek \phi$.  Consider an arbitrary
	subset $\{\Mat A^{m_\ell} \Vek \phi\}_{\ell=0}^{d-1}$ of
	$\{\Mat A^\ell \Vek \phi\}_{\ell=0}^{L-1}$ with
	$\Vek m=(m_0,\dots,m_{d-1})^T$. Then
	the matrix
	\begin{equation*}
		[\Mat A^{m_0} \Vek \phi | \Mat A^{m_1} \Vek \phi | \dots | \Mat
		A^{m_{d-1}} \Vek \phi ] 
		= \Mat S \diag(\Vek \psi) \,
		\Mat V_{\Vek \lambda, \Vek m}.
	\end{equation*}
	is invertible if and only if all elements of $\Vek \psi$ are
	non-zero and if $V_{\Vek \lambda, \Vek m}$ is invertible, which
	means that $\Mat V_{\Vek \lambda}$ has full spark. \qed
\end{proof}

The following result specializes Theorem
\ref{thm:dyn-full-spark-frame} for $\Mat A$ with eigenvalues $\lambda_k=\lambda^k$.

\begin{corollary}\label{thm:full-spark-diag}
  Let $\Mat A \in \BC^{d \times d}$ be diagonalizable with eigenvalues
  $\Vek \lambda = (\lambda^k)_{k=0}^{d-1}$ with $\lambda^k \ne
  1$ for some $\lambda \in \BC$. For every
  $L \ge d$, the set $\{\Mat A^\ell \Vek \phi\}_{\ell=0}^{L-1}$
  is a full spark frame if and only if $\Vek \phi$ depends on
  all eigenvectors.
\end{corollary}

\begin{proof}
	For the chosen $\lambda$, every $d\times d$ sub-matrix of 
	$\Mat V_{\Vek\lambda}$  is an invertible
	Vandermonde matrix.  \qed
\end{proof}

\begin{example}
  Let $L\geq d$ and $\lambda = \e^{2\pi \I/L}$ be the $L$th unit
  root. Consider the matrix $\Mat A = \diag(\lambda^0,\dots,\lambda^{d-1})$ and
  $\Vek{\phi}=\Vek 1$. Then the set
  $\{\Mat A^\ell \Vek \phi \}_{\ell=0}^{L-1}$ is a frame for $\BC^d$
  and has full spark by Corollary~\ref{Cor:dynamical frame} and
  Corollary~\ref{thm:full-spark-diag}.  This frame is called harmonic and
  is related to a submatrix of the discrete Fourier matrix. In
  general not every submatrix of the discrete Fourier transform matrix
  forms a full-spark frame. For more information we refer to
  \cite{alexeev2012full}. 
\end{example}

\begin{theorem}
	Let $L \ge d$, and let $\Mat A\in \BC^{d\times d}$ be
        diagonalizable with distinct real and
	non-negative eigenvalues.  Then
	$\{\Mat A^\ell \Vek \phi \}_{\ell=0}^{L-1}$ is a full-spark frame if
	$\Vek \phi$ depends on all eigenspaces.
\end{theorem}

\begin{proof}
  Due to Theorem~\ref{thm:dyn-full-spark-frame}, the set
  $\{\Mat A^\ell \Vek \phi \}_{\ell=0}^{L-1}$ has full spark if and
  only if the generalized Vandermonde matrices
  $\Mat V_{\Vek \lambda, \Vek m}$ are invertible for every
  $\Vek m \in \BN_0^d$ with distinct coordinates.  Since the Schur
  functions in \eqref{eq:gen-vand-det} have only non-negative
  coefficients, the generalized Vandermonde determinant is here positive
  for all $\Vek \lambda$ with non-negative, distinct coordinates,
  which establishes the assertion.  \qed
\end{proof}

\section{Phase Retrieval in Dynamical Sampling}
\label{sec:phase-retr-dynam}

As mentioned in the introduction, the complementary
property can be exploited to ensure phases retrieval for real signals \cite{aldroubi2020phaseless,aldroubi2017phase}
Since this approach fails in the complex setting, we align $\Vek \phi$
with further sampling vectors allowing polarization.  This allow us to
recover the frame coefficient
$\langle \Vek x ,\Mat A^\ell \Vek \phi\rangle $ up to global phase and then
using the frame property we can reconstruct $\Vek x$.

\begin{theorem}\label{thm: 3107a}
	Let $\{\Mat A^\ell \Vek \phi\}_{\ell=0}^{L-1}$ be a frame for
	$\BC^d$, and let $\alpha_1, \alpha_2 \in \BR$ be real numbers with
	$\alpha_1 - \alpha_2 \not\in \pi\BZ$.  Then almost all
	$\Vek x \in \BC^d$ can be recovered from
	\begin{equation*}
		\bigl\{ |\langle \Vek x, \Mat A^\ell \Vek \phi \rangle|
		\bigr\}_{\ell=0}^{L-1} \cup \bigl\{ | \langle \Vek x,
		\Mat A^\ell (\Vek \phi 
		+\e^{\I \alpha_k} \Mat A \Vek \phi) \rangle | \bigr\}_{\ell=0,k=1}^{L-2,2}
	\end{equation*}
	up to global phase.
\end{theorem}

\begin{proof}
  We consider the dense set of  $\Vek x\in \BC^d$ for which $\langle \Vek x , \Mat A^\ell \Vek \phi \rangle \ne 0$
  for $\ell = 0 , \dots, L-1$. Using the polarization in
  Theorem~\ref{thm:polarization}, we determine the products
  \begin{equation*}
    \overline{\langle \Vek x , \Mat A^\ell \Vek \phi \rangle }
    \langle \Vek x , \Mat A^{\ell+1} \Vek \phi \rangle
    \qquad (\ell = 0,\dots, L-2).
  \end{equation*}
  Considering the phase of the above identity, we calculate the
  relative phases
  \begin{equation*}
    \arg \, \iProdn{\Vek x}{\Mat A^{\ell + 1} \Vek \phi} - \arg \,
    \iProdn{\Vek x}{\Mat A^\ell \Vek \phi}
    \mod 2\pi
    \qquad
    (\ell = 0, \dots, L-2).
  \end{equation*}
  Choosing the phase of $\iProdn{\Vek x}{\Vek \phi}$ arbitrary,  we
  may thus recover the frame coefficients $\iProdn{\Vek x}{\Mat A^\ell
    \Vek \phi}$ up to global phase and thus $\Vek x$.    \qed
\end{proof}

\begin{corollary}
	If $\{\Mat A^\ell \Vek \phi\}_{\ell=0}^{L-1}$ is a frame for
	$\BR^d$, and $\alpha\in \{-1,1\}$, then almost every 
	$\Vek x \in \BR^d$ can be recovered from
	\begin{equation*}
		\bigl\{ |\langle \Vek x, \Mat A^\ell \Vek \phi \rangle|
		\bigr\}_{\ell=0}^{L-1} \cup \bigl\{ | \langle \Vek x,
		\Mat A^\ell (\Vek \phi 
		+\alpha \Mat A \Vek \phi) \rangle | \bigr\}_{\ell=0}^{L-2}
	\end{equation*}
	up to sign.
\end{corollary}

Although the extended measurement set allows the extraction of
relative phases, the proposed procedure may fail in rare cases, where
some of the coefficients $|\langle x , \Mat A^\ell\Vek\phi \rangle |$  are zero for some $\ell$ which means that we are not able to recover $\Vek
x$ if it lies in the union of finitely many hyperplanes. On the contrary,
if the generated frame has full-spark, one do not need all of the
coefficients to recover the wanted signal.

\begin{theorem}
	\label{thm:min-frame-length}
	Let $\{\Mat A^\ell \Vek \phi\}_{\ell=0}^{L-1}$ be a full-spark frame
	for $\BC^d$, and let $\alpha_1, \alpha_2 \in \BR$ be real numbers
	with $\alpha_1 - \alpha_2 \not\in \pi \BZ$.  If
	$L \ge \nicefrac{d^2}{4} + \nicefrac d2$, then every $\Vek x\in\BC^d$
	can be recovered from the samples
	\begin{equation*}
		\bigl\{ |\langle \Vek x, \Mat A^\ell \Vek \phi \rangle|
		\bigr\}_{\ell=0}^{L-1}
		\cup \bigl\{ | \langle \Vek x, \Mat A^\ell (\Vek \phi + \e^{\I \alpha_k}
		\Mat A \Vek \phi) \rangle | \bigr\}_{\ell=0, k=1}^{L-2,2}
	\end{equation*}
	up to global phase.
      \end{theorem}
      
\begin{proof}
	Since $\{\Mat A^\ell \Vek \phi\}_{\ell=0}^{L-1}$ is a full-spark
	frame, we only need to know the phase of $d$ coefficients
	$\langle \Vek x, \Mat A^\ell \Vek \phi \rangle$ to recover $\Vek x$.
	Obviously, if at least $d$ coefficients are zero, then the unknown
	signal is zero everywhere.
	Now assume that  $m<d$ measurements
	$\absn{\langle \Vek x, \Mat A^\ell \Vek \phi \rangle}$ are zero.  As
	soon as we find $d-m$ consecutive non-zero measurements, we can
	transfer the relative phases to enough frame elements to recover
	$\Vek x$ using the extended measurement set.  In the worst case, we
	measure $d-m-1$ consecutive non-zeros followed by a zero.  After
	this pattern has been repeated $m$ times, the remaining measurements
	have to be non-zero. If we thus have at
	least $L\geq (m+1) (d-m)$ measurements, the existence of at least
	$d-m$ non-zero consecutive measurements is guaranteed.  Considering
	that the maximum over $(m+1)(d-m)$ is attained at
	$m \coloneqq \nicefrac{(d-1)}2$ for odd $d$ and
	$m \coloneqq \nicefrac d2$ for even $d$ finishes the proof.  \qed
\end{proof}

\begin{theorem}
	Let $\{\Mat A^\ell \Vek \phi\}_{\ell=0}^{L-1}$ be a full-spark frame
	for $\BC^d$, let $\alpha_1, \alpha_2 \in \BR$ be real numbers
	with $\alpha_1 - \alpha_2 \not\in \pi \BZ$, and let $J \in
        \{0,\dots, d-2 \}$.  If
	$L \ge \nicefrac{(d+1)^2}{4(J+1)}  + d,$
        then every $\Vek x\in\BC^d$ can be recovered from the samples
	\begin{equation*}
		\bigl\{ |\langle \Vek x, \Mat A^\ell \Vek \phi \rangle|
		\bigr\}_{\ell=0}^{L-1}
		\cup \bigl\{ | \langle \Vek x, \Mat A^\ell (\Vek \phi + \e^{\I \alpha_k}
		\Mat A^j \Vek \phi) \rangle | \bigr\}_{\ell=0, k=1, j=1}^{L-2,2,J+1}
	\end{equation*}
	up to global phase.
\end{theorem}

\begin{proof}
	The difference to the proof of Theorem~\ref{thm:min-frame-length} is that
	we may here jump over $J$ consecutive zeros while calculating the
	relative phases.  Thus, if $m$ measurements
	$\absn{\langle \Vek x, \Mat A^\ell \Vek \phi \rangle}$ are zero, the
	worst case scenario is that $d-m-J-1$ consecutive non-zero
	measurements are followed by $J+1$ zeros.  Repeating this pattern
	$\lfloor \nicefrac m{(J+1)} \rfloor$ times, and placing the
        remaining $m \Mod\ (J+1) \le m$
	zeros and $d-m$ non-zeros at the end -- so at most $d$
        elements, we require at most
	\begin{equation*}
		\bigl\lfloor \tfrac m{J+1} \bigr\rfloor (d-m) + (d -
                m) + m \Mod\ (J+1)
		\le \tfrac m{J+1} (d-m) + d
	\end{equation*}
	measurements to transfer the relative phases far enough to
	recover $\Vek x$.  The maximum on the right-hand side is attained at
	$m \coloneqq \nicefrac{(d+1)}2$ for odd $d$ and
	$m \coloneqq \nicefrac d2$ for even $d$, which finishes the proof.
	\qed
\end{proof}

\bibliographystyle{abbrv}
{\footnotesize \bibliography{literature}}

\end{document}